\documentclass{article}
\usepackage{graphicx,amsmath,amssymb,amsfonts,authblk}

\usepackage{amsthm}
\usepackage{paralist}
\usepackage{ifthen}
\usepackage{tikz}

\date{}

\newtheorem{theorem}{Theorem}[section]
\newtheorem{lemma}[theorem]{Lemma}
\newtheorem{proposition}[theorem]{Proposition}

\theoremstyle{remark}
\newtheorem{remark}[theorem]{Remark}

\newcommand{\parder}[3][Default]{
	\frac{\partial \ifthenelse{\equal{#1}{Default}}{}{^{#1}}#2}{
              \partial #3 \ifthenelse{\equal{#1}{Default}}{}{^{#1}}}}

\numberwithin{equation}{section}

\makeatletter
\newcommand{\nolisttopbreak}{\vspace{\topsep}\nobreak\@afterheading}
\makeatother

\newcommand{\jac}{{\mathcal{J}}}

\newcommand{\Mat}{\operatorname{Mat}}

\newcommand{\GL}{\operatorname{GL}}
\newcommand{\rk}{\operatorname{rk}}

\newcommand{\Z}{\mathbb{Z}}

\newcommand{\tr}{\operatorname{tr}}
\newcommand{\trdeg}{\operatorname{trdeg}}
\newcommand{\chr}{\operatorname{char}}

\begin{document}

\title{Classification of cubic homogeneous polynomial maps with Jacobian matrices of rank two}

\author{Michiel de Bondt\footnote{M.deBondt@math.ru.nl}}
\affil{\small Institute for Mathematics, Astrophysics and Particle Physics,
Radboud University Nijmegen, The Netherlands}
\author{Xiaosong Sun\footnote{Corresponding author, E-mail: sunxs@jlu.edu.cn}}
\affil{\small School of Mathematics, Jilin University, Changchun 130012, China}

\maketitle

\begin{abstract}
Let $K$ be any field with $\textup{char}K\neq 2,3$. We classify all cubic homogeneous polynomial maps $H$ over $K$ with $\textup{rk} JH\leq 2$. In particular, we show that, for such an $H$, if $F=x+H$ is a Keller map then $F$ is invertible, and furthermore $F$ is tame if the dimension $n\neq 4$.
\end{abstract}

\section{Introduction}

Let $K$ be an arbitrary field and $K[x]:=K[x_1,x_2,\ldots,x_n]$  the polynomial ring in $n$ variables.
For a polynomial map $F=(F_1,F_2,\ldots,F_m)\in K[x]^m$, we denote by
$\jac F:=(\frac{\partial F_i}{\partial x_j})_{m\times n}$ the Jacobian matrix of $F$ and $\deg F:=\max_i \deg F_i$ the degree of $F$.  A polynomial map $H\in K[x]^m$ is called homogeneous of degree $d$ if each $H_i$ is zero or homogeneous of degree $d$.

A polynomial map $F\in K[x]^n$ is called a Keller map if $\det \jac F\in K^*$. The Jacobian conjecture asserts that any Keller map is invertible if $\textup{char} K=0$; see \cite{essen2000} or \cite{bass1982}.
It is still open for any dimension $n\geq 2$.

Following \cite{shestakov2}, we call a polynomial automorphism elementary if 
it is of the form $(x_1,\ldots,x_{i-1},cx_i+a,x_{i+1},\ldots,x_n)$,
where $c\in K^*$ and $a\in K[x]$ contains no $x_i$. Furthermore, we call
a polynomial automorphism tame if it is a finite composition of elementary ones. 
The definitions of elementary and tame may be different in other sources, but
(as long as $K$ is a generalized Euclidean ring) the definitions of tame are 
equivalent. The Tame Generators Problem asks if every
polynomial automorphism is tame. It has an affirmative
answer in dimension 2 for arbitrary characteristic (see \cite{jung, kulk}) and
a negative answer in dimension 3 for the case of $\textup{char} K=0$ (see \cite{shestakov2}), and is still
open for any $n\geq 4$.

A polynomial map $F=x+H\in K[x]^n$ is called triangular if $H_n\in K$ and $H_i\in
K[x_{i+1},\ldots,x_n],$ $1\leq i \leq n-1$. A
polynomial map $F$ is called linearly triangularizable if it is
linearly conjugate to a triangular map, i.e., there exists an
invertible linear map $T\in \GL_n(K)$ such that $T^{-1}F(Tx)$ is triangular. A
linearly triangularizable map is tame.

Some special polynomial maps have been investigated in the literature. For example, when $\textup{char}K=0$, a Keller map $F=x+H\in K[x]^n$ is shown to be linearly triangularizable in the cases: (1) $n=3$ and $H$ is homogeneous of arbitrary degree $d$ (de Bondt and van den Essen \cite{bondt05}); (2) $n=4$ and $H$ is quadratic homogeneous (Meisters and Olech \cite{mei91}); (3) $n=9$ and $F$ is a quadratic homogeneous quasi-translation (Sun \cite{sun10}); (4) $n$ arbitrary and $H$ is quadratic with $\rk \jac H\leq 2$ (De Bondt and Yan \cite{bondt-yan}),
and to be tame in the case (5) $n=5$ and $H$ is quadratic homogeneous (de Bondt \cite{bondt09} and Sun \cite{sun} independently), and to be invertible in the case (6) $n=4$ and $H$ is cubic homogeneous (Hubbers \cite{hub94}).
For the case of arbitrary characteristic, de Bondt \cite{bondt17} described the Jacobian matrix $\jac H$ of rank two for any quadratic polynomial map $H$ and showed that if $\jac H$ is nilpotent then $\jac H$ is similar to a triangular one.

In this paper, we investigate cubic homogeneous polynomial maps $H$ with $\textup{rk} \jac H\leq 2$ for any dimension $n$ when $\textup{char}K\neq 2, 3$. In Section 2, we classify all such maps (Theorem \ref{rkle2}). And in Section 3, we show that for such an $H$, if $F=x+H$ is a Keller map, then it is invertible and furthermore it is tame if the dimension $n\neq 4$ (Theorem \ref{uporkle2}).

\section{Cubic homogeneous maps $H$ with $\textup{rk} JH\leq 2$}

For a polynomial map $H\in K[x]^m$, we write $\trdeg_K K(H)$ for the transcendence degree of $K(H)$ over $K$.
It is well-known that $\rk \jac H = \trdeg_K K(H)$~ if $K(H) \subseteq K(x)$ is separable, in particular if
$\textup{char}K=0$; see \cite[Proposition 1.2.9]{essen2000}. And
for arbitrary characteristic, one has $\rk \jac H \leq \trdeg_K K(H)$; see \cite{bondt15} or  \cite{pss}.

It was shown in \cite{bondt17} that when $\textup{char}K\neq 2$, for any quadratic polynomial map $H$ with $\rk \jac H\leq 2$, one has $\rk \jac H = \trdeg_K K(H)$. We will show that when $\textup{char}K\neq 2,3$, for any cubic homogeneous polynomial map $H$ with $\rk \jac H\leq 2$, one has $\rk \jac H = \trdeg_K K(H)$. The notation $a|_{x=c}$ below means to substitute $x$ by $c$ in $a$.

\begin{theorem} \label{detdep}
Let $s \le n$. Take
$$
\tilde{x} := (x_1,x_2,\ldots,x_s) ~~ \mbox{and} ~~
L := K(x_{s+1},x_{s+2},\ldots,x_n)\mbox{.}
$$
To prove that for (homogeneous) polynomial maps
$H \in K[x]^m$ of degree $d$,
\begin{equation} \label{eq1}
\rk \jac H = r ~~\mbox{implies } \trdeg_K K(H) = r,~
\mbox{~for every } r < s\mbox{,}
\end{equation}
it suffices to show that for (homogeneous) polynomial maps
$\tilde{H} \in L[\tilde{x}]^s$ of degree $d$,
\begin{equation} \label{eq2}\textup{trdeg}_LL(\widetilde{H})=s ~~ \mbox{implies}~~
\rk \jac_{\tilde{x}} \tilde{H}=s\mbox{.}
\end{equation}
\end{theorem}

\begin{proof}
Suppose that $H \in K[x]^m$ is (homogeneous) of degree $d$, such that
\eqref{eq1} does not hold. Then there exists an $r < s$ such that
$\rk \jac H = r < \trdeg_K K(H)$. We need to show that \eqref{eq2} does not hold.

Let $s'= \trdeg_K K(H)$. Assume without
loss of generality that $H_1, H_2, \ldots, H_{s'}$ are
algebraically independent over $K$, and that
the components of
$$
H' := \big(H_1, H_2, \ldots, H_{s'},
      x_{s'+1}^d, x_{s'+2}^d,\ldots,x_s^d\big)
$$
are algebraically independent over $K$ if $s'< s$. Then
$$
\rk \jac H' \le r + (s - s') < s = \trdeg_K K(H')\mbox{.}
$$

For the case of $s'\geq s$, just take $H'=(H_1,H_2,\ldots,H_s)$, and we have also $
\rk \jac H' \le r < s.
$

Notice that \eqref{eq1} is also unsatisfied for $H'$. So, replacing
$H$ by $H'$, we may assume that $H\in K[x]^s$ with $\rk \jac H=r<\trdeg_K K(H) = s$.

One may observe that
$H_1(x_1,x_1x_2,x_1x_3,\ldots,x_1x_n)$ is algebraically
independent over $K$ of $x_2,x_3,\allowbreak\ldots,x_n$.
On account of the Steinitz Mac Lane exchange lemma, we may assume without
loss of generality that the components of
$$
\big(H(x_1,x_1x_2,x_1x_3,\ldots,x_1x_n),x_{s+1},x_{s+2},\ldots,x_n\big)
$$
are algebraically independent over $K$. Then the components of
$H(x_1,x_1x_2,x_1x_3,$ $\ldots,x_1x_n)$ are algebraically independent over $L:= K(x_{s+1},x_{s+2},\ldots,x_n)$, and so are the components of
$$
\tilde{H} := H(x_1,x_2,\ldots,x_s,
x_1x_{s+1},x_1x_{s+2},\ldots,x_1x_n)\in L[\widetilde{x}]^s\mbox{,}
$$
where $\widetilde{x}=(x_1,x_2,\ldots,x_s)$.
That is, $\textup{trdeg}_LL(\widetilde{H})=s$.

Let $G := (x_1,x_2,\ldots,x_s, x_1x_{s+1},x_1x_{s+2},\ldots,x_1x_n)$.
Then it follows from the chain rule that
$$
\jac_{\tilde{x}} \tilde{H} = (\jac H)|_{x = G} \cdot \jac_{\tilde{x}} G\mbox{,}
$$
so $\rk \jac_{\tilde{x}} \tilde{H}\leq \rk (\jac H)|_{x = G}\leq \rk\jac H<s$. Therefore \eqref{eq2} does not hold for $\widetilde{H}$, which
completes the proof.
\end{proof}


\begin{lemma} \label{rkform}
Let $H \in K[x]^m$ be a polynomial map of degree $d$ and $r := \rk \jac H$. Denote by $|K|$ the cardinality of $K$.
\begin{enumerate}[\upshape (i)]

\item If $|K|>(d-1)r$ and $\jac H \cdot x = 0$,
then there exist $S \in \GL_m(K)$ and $T \in \GL_n(K)$,
such that for $\tilde{H} := SH(Tx)$,
$$
\tilde{H}|_{x=e_{r+1}} = \left(\begin{array}{cc}
I_r & 0 \\ 0 & 0
\end{array} \right)\mbox{.}
$$

\item If $|K|>(d-1)r + 1$ and $\jac H \cdot x \ne 0$,
then there exist $S \in \GL_m(K)$ and $T \in \GL_n(K)$,
such that for $\tilde{H} := SH(Tx)$,
$$
\tilde{H}|_{x=e_1} = \left( \begin{array}{cc}
I_r & 0 \\ 0 & 0
\end{array} \right)\mbox{.}
$$

\end{enumerate}
Moreover, $|K|$ may be one less (i.e. at least
$(d-1)r$ and $(d-1)r + 1$ respectively) if every nonzero component of
$H$ is homogeneous.
\end{lemma}

\begin{proof}
(i) Assume without loss of generality that
$$
a_0:= \det \jac_{x_1,x_2,\ldots,x_r} (H_1,H_2,\ldots,H_r) \ne 0\mbox{.}
$$
Suppose that $|K|>(d-1)r$.
It follows by \cite[Lemma 5.1 (i)]{bondt13} that there exists a
vector $w \in K^n$ such that $a_0(w) \ne 0$.
So $\rk \big(\jac H\big)\big|_{x=w} = r$.
There exist $n-r$
independent vectors 
$v_{r+1}, v_{r+2}, \ldots, \allowbreak v_n \in K^n$,
such that
$
\big(\jac H\big)\big|_{x=w} \cdot v_i = 0
$
for $i = r+1, r+2, \allowbreak \ldots, n$. And we may take $v_{r+1} = w$ since
$$
\big(\jac H\big)\big|_{x=w} \cdot w = \big(\jac H \cdot x\big)\big|_{x=w} = 0\mbox{.}
$$
Take $T =(v_1,v_2,\cdots,v_n)\in \GL_n(K)$.
From the chain rule, we deduce that
$$
\big(\jac (H(Tx))\big)\big|_{x=e_{r+1}} \cdot e_i
= (\jac H)|_{x=Te_{r+1}} \cdot T e_i = (\jac H)|_{x=w} \cdot v_i
\quad (1\leq i\leq n)\mbox{.}
$$
In particular,
$\rk \jac (H(Tx))\big|_{x=e_{r+1}} = r$ and the last
$n - r$ columns of $\big(\jac (H(Tx))\big)\big|_{x=e_{r+1}}$ are zero.
There exists $S \in \GL_m(K)$ such that
$$
\big(\jac (SH(Tx))\big)\big|_{x=e_{r+1}} =
S \cdot \big(\jac (H(Tx))\big)\big|_{x=e_{r+1}} =
\left( \begin{array}{cc}
I_r & 0\\ 0 & 0
\end{array} \right)\mbox{.}
$$

(2) Suppose that $|K|>(d-1)r + 1$.
Since $\jac H \cdot x \ne 0$, we may assume that
$$
\rk \Big( \jac H \cdot x,
\jac_{x_2,x_3,\ldots,x_r} H \Big) = r\mbox{,}
$$
and that
$$
a_1 := \det \Big( \jac (H_1,H_2,\ldots,H_r) \cdot x,
\jac_{x_2,x_3,\ldots,x_r} (H_1,H_2,\ldots,H_r) \Big) \ne 0\mbox{.}
$$
It follows by \cite[Lemma 5.1 (i)]{bondt13} that there exists $w \in K^n$ such that $a_1(w) \ne 0$.
One may observe that $\rk \big(\jac H\big)\big|_{x=w} = r$
and thus there exist independent vectors $v_{r+1}, v_{r+2}, \ldots, v_n \in K^n$,
such that $\big(\jac H\big)\big|_{x=w} \cdot v_i = 0$
for $i = r+1, \allowbreak r+2, \ldots, n$. Since
$\big(\jac H \cdot x\big)\big|_{x=w}$ is the first column of a full column rank matrix, we have
$$
\big(\jac H\big)\big|_{x=w} \cdot w = \big(\jac H \cdot x\big)\big|_{x=w} \ne 0\mbox{.}
$$
So $v_1 := w$ is independent of $v_{r+1}, \allowbreak v_{r+2}, \ldots, v_n$.

Take $T =(v_1,v_2,\cdots,v_n)\in \GL_n(K)$. Then
$$
\big(\jac (H(Tx))\big)\big|_{x=e_1} \cdot e_i
= (\jac H)|_{x=Te_1} \cdot T e_i = (\jac H)|_{x=w} \cdot v_i
\quad (1\leq i\leq n).
$$
The rest of the proof of (ii) is similar to
that of (i).

The last claim follows from \cite[Lemma 5.1 (ii)]{bondt13}, as an
improvement to \cite[Lemma 5.1 (i)]{bondt13}.
\end{proof}

\begin{proposition} \label{propA}
Assume that $\textup{char}K\notin \{1,2,\ldots,d\}$. Then for any cubic homogeneous polynomial map $H\in K[x]^m$ of degree $d$ with $\rk \jac H\leq 1$, the components of $H$ are linearly 
dependent over $K$ in pairs, and one has $\rk \jac H = \trdeg_K K(H)$.
\end{proposition}


\begin{proof} The case $\rk \jac H=0$ is obvious, so let 
$\rk \jac H=1$. On account of Lemma \ref{rkform}, we may assume that 
$\jac H|_{x=e_1}=E_{11}$. Let $j \ge 2$.
Since $\deg_{x_1} H_j < d$, we infer that either $H_j = 0$, 
or $\deg_{x_1} \parder{}{x_1} H_j < \deg_{x_1} \parder{}{x_i} H_j$ 
for some $i \ge 2$, where $\deg_{x_1} 0 = -\infty$. 
The latter is impossible due to $\rk \jac H=1$, so $H_j = 0$.
This holds for all $j \ge 2$, which yields the desired results.
\end{proof}

\begin{lemma}\label{lemB} Let $H=(h,x_1^2x_2,x_2^2x_3)$ or $(h,x_1^2x_3,x_2^2x_3)\in K[x_1,x_2,x_3]^3$, where $h$ is cubic homogeneous, and assume that $\chr K\neq 2,3$. Then $\rk \jac H=\textup{trdeg}_KK(H)$.
\end{lemma}

\begin{proof} It suffices to consider the case of $\rk \jac H=2$. Define a derivation $D$ on $A=K[x_1,x_2,x_3]$ as follows: for any $f\in A$, $$D(f)=\frac{x_1x_2x_3}{H_2H_3}\det \jac H\mbox{.}$$
In the case $H=(h,x_1^2x_2,x_2^2x_3)$, an easy calculation shows that
$D=x_1\partial_{x_1}-2x_2\partial_{x_2}+4x_3\partial_{x_3}$. Then for any term $u=x_1^{d_1}x_2^{d_2}x_3^{d_3}\in A$, $D(u)=(d_1-2d_2+4d_3)u$. And thus $\ker D:=\{g\in A \mid D(g)=0\}$, the kernel of $D$, is linearly spanned by all terms $u$ with $d_1-2d_2+4d_3=0$. So the only cubic terms in $\ker D$ are $x_1^2x_2$ and $x_2^2x_3$. Since $\rk \jac H=2$, we have $\det\jac H=0$ and thus $h\in \ker D$, which implies that $h$ is a linear combinations of  $x_1^2x_2$ and $x_2^2x_3$. Thus $\textup{trdeg}_KK(H)=2$.

In the case $H=(h,x_1^2x_3,x_2^2x_3)$, one may verify that 
$x_1^2 x_3$, $x_1 x_2 x_3$ and $x_2^2 x_3$ are the
only cubic terms in $\ker D$. The conclusion follows similarly.
\end{proof}

\begin{theorem} \label{eqthm} Assume that $\textup{char}K\neq 2,3$. Then for any cubic homogeneous polynomial map $H\in K[x]^m$ with $\rk \jac H\leq 2$, one has $\rk \jac H = \trdeg_K K(H)$.
\end{theorem}

\begin{proof} Due to Theorem \ref{detdep}, and replacing $L$ there by $K$,
we may assume that $H\in K[x_1,x_2,x_3]^3,$
and it suffices to show that
$$\textup{trdeg}_KK(H)=3 ~ ~\mbox{implies}~
\rk \jac H=3\mbox{,}$$ or equivalently,
\begin{equation} \label{detdep3}
\det \jac H = 0 ~\mbox{implies} ~
\textup{trdeg}_KK(H)<3\mbox{.}
\end{equation}
So assume that $\det \jac H=0$. Since we may replace $K$ by an extension field to make it large
enough, it follows by Lemma \ref{rkform} that we may assume that
$
\big(\jac H\big)\big|_{x=e_1}
=E_{11}+E_{22}.
$ Then $\jac H$ is of the form 
$$\left(
   \begin{array}{ccc}
     x_1^2+\ast & \ast & \ast \\
     \ast & x_1^2+\ast & \ast \\
     \ast & \ast & \frac{\partial H_3}{\partial x_3} \\
   \end{array}
 \right)\mbox{,}
$$ where the $x_1$-degree of each element $\ast$ is less than 2. Observing the terms with $x_1$-degree $\geq 5$ in $\det \jac H$, we have that
$\frac{\partial H_3}{\partial x_3}\in K[x_2,x_3]$. Notice that $H_2$ and $H_3$ are of the form:
\begin{align*}
H_2&=x_1^2x_2+b_{10}x_1x_3^2+b_{11}x_1x_2x_3+b_{12}x_1x_2^2+b_0(x_2,x_3)\mbox{;}\\
H_3&=c_{12}x_1x_2^2+c_{00}x_3^3+c_{01}x_2x_3^2+c_{02}x_2^2x_3+c_{03}x_2^3\mbox{.}
\end{align*}
We shall show that $x_2^2 \mid H_3$, i.e., $c_{00} = c_{01} = 0$.

Noticing that the part of $x_1$-degree 4 of $\det \jac H$ is
$\big(\frac{\partial H_3}{\partial x_3}-\frac{\partial H_2}{\partial x_1\partial x_3}\frac{\partial H_3}{\partial x_1\partial x_2}\big)x_1^4$, we see that
$\frac{\partial H_3}{\partial x_3}-\frac{\partial H_2}{\partial x_1\partial x_3}\frac{\partial H_3}{\partial x_1\partial x_2}=0$.
Consequently,
$$
(3 c_{00}x_3^2 + 2 c_{01}x_2 x_3 + c_{02}x_2^2) =
(2 b_{10} x_3 + b_{11} x_2) (2 c_{12} x_2)
$$
so
\begin{align*}
c_{00} &= 0 & c_{01} &= 2 b_{10} c_{12} & c_{02} &= 2 b_{11} c_{12}
\end{align*}
One may observe that the coefficient of $x_1^3x_3^3$ in 
$\det \jac H$ is $2c_{01}b_{10}=0$, which we can combine with
$c_{01} = 2 b_{10} c_{12}$ to obtain $c_{01} = 0$.
Therefore, 
$$H_3=(c_{12}x_1+c_{03}x_2+c_{02}x_3)x_2^2\mbox{.}$$
Moreover, if $c_{12}=0$ then $c_{02} = 2 b_{11} c_{12} = 0$ and thus $H_3=c_{03} x_2^3$. 

We distinguish two cases.
\begin{itemize}

\item\emph{Case 1:} $c_{12} \ne 0$ and 
$c_{12}x_1+c_{03}x_2+c_{02}x_3 \nmid H_i$ 
for some $i$.

Then $H_3$ is the product of two linear forms, of which
two are distinct. Hence we can compose $H$ with invertible
linear maps on both sides, to obtain a map $H'$ for which 
$H'_2 = x_1^2 x_2$, and $x_2\nmid H'_1$.

Notice that $H'_1(1,0,t) \ne 0$. As $K$ has
at least $5$ elements, it follows from [3, Lemma 5.1 (i)] that there
exists a $\lambda \in K$, such that $H'_1(1,0,\lambda) \ne 0$.
Hence the coefficient of $x_1^3$ in $H'_1(x_1,x_2,x_3+\lambda x_1)$ 
is nonzero. Furthermore, $H'_2(x_1,x_2,x_3+\lambda x_1) = x_1^2 x_2$.

Replacing $H'$ by $H'(x_1,x_2,x_3+\lambda x_1)$, we may assume that 
$H'_2=x_1^2x_2$ and that $H'_1$ contains $x_1^3$ as a term.
We may even assume that the coefficient of $x_1^3$ in $H'_1$
equals $1$. Then $\jac H'|_{x=e_1}$ is of the form 
$$
\left( \begin{array}{ccc}
1 & \ast & a \\
0 & 1 & 0 \\
\ast & \ast & \ast \\
\end{array} \right)\mbox{,}
$$
and has rank 2. Furthermore, $v_3=(-a,0,1)^t$ belongs to its null space.  
We may apply the proof of Lemma \ref{rkform} on $H'$ by taking 
$T=(e_1,e_2,v_3)$ and taking an appropriate $S\in \GL_3(K)$ such 
that $\widetilde{H}:=SH'(Tx)$ satisfies
$\jac \widetilde{H}|_{x=e_1}=S\jac H'|_{x=Te_1}T=E_{11}+E_{22}$.
Notice that $Tx$ is of the form $(L_1,x_2,L_3)$, and observing 
the form of $\jac H'|_{x=e_1}$ one may also choose $Sx$ to be 
of the form $(\ast, x_2,\ast)$. Then $\widetilde{H}_2=L_1^2x_2$. 

So we can compose $\widetilde{H}$ with an invertible linear map 
on the right, to obtain a map $\widetilde{H}'$ for which 
$\widetilde{H}'_2 = x_1^2 x_2$ and $\widetilde{H}'_3 = x_2^2 L'$ 
for some linear form $L'$. 

Suppose first that $L'$ is a linear combination of $x_1$ and $x_2$.
If $\widetilde{H}'_1 \in K[x_1,x_2]$, then we are done. Otherwise,
we have $\det \jac_{x_1,x_2} (\widetilde{H}'_2,\widetilde{H}'_3)=0$, 
and then by Proposition \ref{propA}, 
$\textup{trdeg}_K K(H'_2,H'_3)<2$. 

Suppose next that $L'$ is not a linear combination of 
$x_1$ and $x_2$. Then we may assume that $\widetilde{H}'_3 = x_2^2 x_3$. 
By Lemma \ref{lemB} (i), $\textup{trdeg}_K K(\widetilde{H}')<3$.

\item\emph{Case 2:} $c_{12} = 0$ or 
$c_{12}x_1+c_{03}x_2+c_{02}x_3 \mid H_i$ 
for all $i$.

Since $x_2^2 \mid H_3$, we can compose $H$ with invertible
linear maps on both sides, to obtain a map $H'$ for which 
$H'_1 \in \{x_1^3, x_1^2 x_2\}$.
After a possible interchange of $H'_2$ and $H'_3$, 
the first two rows of $\jac H'$ are independent.
Now we may apply the proof of Lemma \ref{rkform} to $H'$,
more precisely, there exist $S,T\in \GL_3(K)$ such that 
$\widetilde{H} := SH'(Tx)$ satisfies 
$\jac \widetilde{H}|_{x=e_1}=E_{11}+E_{22}$. 
If we choose $w$ such that first two rows of $(\jac H')_{x=w}$
are independent, then we can take $S$ such that
$Sx=(f_1x_1+f_2x_2,g_1x_1+g_2x_2,\ast)$. 
By repeating the discussion for $\widetilde{H}$ as for $H$ above, 
we may assume that $\widetilde{H}_3=Lx_2^2$ for some linear form $L$.

Let $Tx=(L_1,L_2,L_3)$. Notice that 
$H'_1(Tx) \in \{L_1^3,L_1^2L_2\}$ and that $H'_1(Tx)$ is a linear 
combination of $\widetilde{H}_1$ and $\widetilde{H}_2$. Hence we can 
compose $\widetilde{H}$ with a linear map on the left, to obtain a
map $\widetilde{H}'$ for which 
$\widetilde{H}'_2 \in \{L_1^3,L_1^2L_2\}$ and 
$\widetilde{H}'_3 = Lx_2^2$.

Suppose first that $\widetilde{H}'_2 = L_1^2L_2$. 
Then $c_{12} \ne 0$, so
$c_{12}x_1+c_{03}x_2+c_{02}x_3 \mid H_i$ for all $i$. 
From this, we infer that $L_2 \mid \widetilde{H}_i$ and
$L_2 \mid \widetilde{H}'_i$ for all $i$. As 
$x_2 \nmid \widetilde{H}_1$, we deduce that $L$ and $L_2$ are
dependent linear forms, which are independent of $x_2$.
If $L$ and $L_2$ are linear combinations of $L_1$ and $x_2$, then we
can reduce to Proposition \ref{propA}, and otherwise we can
reduce to Lemma \ref{lemB} (ii).

Suppose next that $\widetilde{H}'_2 = L_1^3$. If $L$, $L_1$ and $x_2$
are linearly dependent over $K$, then we can reduce to Proposition \ref{propA}.
Otherwise, $\tilde{H}$ is as $H$ in the previous case. \qedhere

\end{itemize}
\end{proof}

\begin{remark} Inspired by Lemma \ref{lemB}, we investigated 
maps $H$ of which the components are terms, and searched for
$H$ with algebraically independent components for which 
$\det \jac H = 0$. One can infer that $H$ is as such, if and 
only if the matrix with entries $\deg_{x_i} H_j$ has determinant 
zero over $K$, but not over $\Z$.

We found the following non-homogeneous $H$ as above over fields of
characteristic $5$:
$$
(x_1^3 x_2, x_1 x_2^2), \quad (x_1^2 x_2, x_1 x_3^2, x_2 x_3)
$$
with the following homogenizations respectively:
$$
(x_1^3 x_2, x_1 x_2^2 x_3, x_3^4), \quad (x_1^2 x_2, x_1 x_3^2, x_2 x_3 x_4, x_4^3)
$$
Besides these homogenizations, we found the following homogeneous $H$
over fields of characteristic $5$:
$$
(x_1^2 x_3^2, x_1 x_2^3, x_2 x_3^3), \quad (x_4 x_1^2, x_1 x_2^2, x_2 x_3^2, x_3 x_4^2)
$$
We conclude with a homogeneous $H$ over fields of characteristic $7$, and
a homogeneous $H$ over any characteristic $p \in \{1,2,\ldots,d\}$ 
respectively:
$$
(x_3 x_1^3, x_1 x_2^3, x_2 x_3^3), \quad (x_1^d, x_1^{d-p} x_2^p)
$$
These examples show that the conditions in Proposition \ref{propA} and Theorem
\ref{eqthm} cannot be relaxed.
\end{remark}

\begin{theorem} \label{rkle2}
Suppose that $\textup{char} K\neq 2,3$ and let $H \in K[x]^m$ be cubic
homogeneous. Let $r := \rk \jac H$ and suppose that $r \le 2$.
Then there exist $S \in \GL_m(K)$ and $T \in \GL_n(K)$,
such that for $\tilde{H} := SH(Tx)$, one of the following statements holds:
\begin{enumerate}[\upshape (1)]

\item $\tilde{H}_{r+1} = \tilde{H}_{r+2} = \cdots = \tilde{H}_m = 0$;

\item $r=2$ and $\tilde{H} \in K[x_1,x_2]^m$;

\item $r=2$ and $K \tilde{H}_1 + K \tilde{H}_2 + \cdots + K \tilde{H}_m
= K x_3 x_1^2 \oplus K x_3 x_1 x_2 \oplus K x_3 x_2^2$.
\end{enumerate}
Furthermore, we may take $S = T^{-1}$ if $m = n$.
\end{theorem}

\begin{proof} By Theorem \ref{eqthm}, $\trdeg_KK(H)=\rk \jac H=r\leq 2$.
Since $H$ is homogeneous, we have
$\trdeg_K K(tH) = r$ as well, where $t$ is a new variable.

Suppose first that $r \le 1$.
It follows by \cite[Theorem 2.7]{bondt15} that we may take $\tilde{H}$ as in (1).

Suppose next that $r = 2$. By \cite[Theorem 2.7]{bondt15},
$H$ is of the form $g \cdot h(p,q)$,
such that $g, h$ and $(p,q)$ are homogeneous and
$\deg g + \deg h \cdot \allowbreak \deg (p,q) = 3$.

If $\deg h \le 1$, then every triple of components of $h$ is linearly dependent
over $K$, and thus we may take $\tilde{H}$ as in (1).
If $\deg h = 3$, then $\deg (p,q) = 1$ and $\deg g = 0$, whence we may take
$\tilde{H}$ as in (2).

So assume that $\deg h = 2$. Then $\deg(p,q) = 1$ and $\deg g = 1$. If
$g$ is a linear combination of $p$ and $q$, then we may take $\tilde{H}$
as in (2). If $g$ is not a linear combination of $p$ and $q$,
then we may take $\tilde{H}$ as in (3) or (1).

Finally, if $m = n$ and $\tilde{H}=SH(Tx)$ is as in (1), then $SH(S^{-1}x)=\tilde{H}(T^{-1}S^{-1}x)$ is still as in (1).
So we may take $S=T^{-1}$. If $m = n$ and $\tilde{H}=SH(Tx)$ is as in (2) or (3), then
$T^{-1}H(Tx)=T^{-1}S^{-1}\tilde{H}$ is still as in (2) or (3), whence we may also take $S=T^{-1}$.
\end{proof}

\section{Cubic homogeneous Keller maps $x+H$ with $\textup{rk} JH\leq 2$}

For two matrices $M, N\in \Mat_n(K[x])$, we say that $M$ is similar over $K$ to $N$, if there exists $T \in \GL_n(K)$ such that
$N = T^{-1}MT$. 

\begin{theorem} \label{trdeg1}
Let $F = x + H\in K[x]^n$ be a Keller map with $\trdeg_K K(H) = 1$.
Then $\jac H$ is similar over $K$ to a triangular matrix, and the following
statements are equivalent:
\begin{enumerate}[\upshape (1)]

\item $\det \jac F = 1$;

\item $\jac H$ is nilpotent;

\item $(\jac H) \cdot (\jac H)|_{x=y} = 0$, where $y=(y_1,y_2,\ldots,y_n)$ are $n$ new variables.

\end{enumerate}
\end{theorem}

\begin{proof}
Since $\trdeg_K K(H) = 1$, by \cite[Corollary 3.2]{bondt15} there exists a polynomial
$p \in K[x]$ such that $H_i \in K[p]$ for each $i$. Say that $H_i = h_i(p)$,
where $h_i \in K[t]$ for each $i$. Write $h_i' = \parder{}{t} h_i$,
then
\begin{equation} \label{JHhp}
\jac H = h'(p) \cdot \jac p\mbox{.}
\end{equation}
Assume without loss of generality that
$$
h_1' = h_2' = \cdots = h_s' = 0\mbox{,}
$$
and that
$$
0 \le \deg h_{s+1}' < \deg h_{s+2}' < \cdots < \deg h_n'\mbox{.}
$$
For  $s < i < n$,
$$
\deg h_i'(p) = \deg h_i' \cdot \deg p
\le (\deg h_{i+1}' - 1) \cdot \deg p = \deg h_{i+1}'(p) - \deg p.
$$
Since the degrees of the entries of $\jac p$ are less than $\deg p$, we deduce from
\eqref{JHhp} that the nonzero entries on the diagonal of $\jac H$
have different degrees in increasing order. Furthermore, the
nonzero entries beyond the $(s+1)$th entry on the diagonal of $\jac H$ have positive
degrees.

By \eqref{JHhp}, $\rk (- \jac H) \le 1$, and thus
$n-1$ eigenvalues of $-\jac H$ are zero. It follows that
the trailing degree of the characteristic polynomial of $-\jac H$ is at least
$n-1$. More precisely,
$$
\det (t I_n+\jac H) = t^n - \tr (-\jac H) \cdot t^{n-1}\mbox{,}
$$
and thus
$$
\det \jac F = \big(t^n - \tr (-\jac H) \cdot t^{n-1}\big)\big|_{t=1} = 1 + \tr \jac H\mbox{.}
$$
Observe that the diagonal of $\jac H$ is totally zero, except maybe the
$(s+1)$th entry, which is a constant.

So $\parder{}{x_i} p = 0$ for all $i > s+1$, and $\jac H$ is lower triangular.
If the $(s+1)$th entry on the diagonal of $\jac H$ is nonzero, then
(1), (2) and (3) do not hold. If the $(s+1)$th entry on the
diagonal of $\jac H$ is zero, then $\parder{}{x_i} p = 0$ for all $i > s$, whence
(1), (2) and (3) hold.
\end{proof}

Let $H \in K[x]^n$ be homogeneous of degree $d\geq 2$. Then $x + H$ is
a Keller map if and only if $\jac H$ is nilpotent; see for example \cite[Lemma 6.2.11]{essen2000}. So we first investigate  nilpotent matrices over $K[x]$.

\begin{lemma} \label{2x2}
Let $N \in \Mat_2(K[x])$ such that $N$ is nilpotent. Then there
exist $a,b,c \in K[x]$ such that
$$
N = c \left( \begin{array}{cc}
ab & -b^2 \\ a^2 & -ab \end{array}\right)\mbox{.}
$$
Furthermore, $N$ is similar over $K$ to a triangular matrix if and only
if $a$ and $b$ are linearly dependent over $K$.
\end{lemma}

\begin{proof}
Since $\det N = 0$, we may write $N$ in the form
$$
N = c \cdot \binom{b}{a} \cdot \big(\,a ~~ {-\tilde{b}}\,\big)\mbox{,}
$$
where $a,b \in K[x]$ and $\tilde{b},c \in K(x)$. Since $\tr N = 0$,
we have $\tilde{b} = b$. If we choose $a$ and $b$ to be relatively prime,
then $c \in K[x]$ as well.

Furthermore, $a$ and $b$ are linearly dependent over $K$
if and only if the rows of $N$ are linearly dependent over $K$,
if and only if $N$ is similar over $K$ to a triangular matrix.
\end{proof}

\begin{lemma} \label{2x2nilp}
Let $H \in K[x]^2$ be cubic homogeneous, such that $\jac_{x_1,x_2} H$ is
nilpotent. Then there exists $T \in \GL_2(K)$ such that for
$\tilde{H} := T^{-1} H\big(T(x_1,x_2),x_3,\allowbreak x_4,\allowbreak
\ldots,x_n\big)$, one of the following statements holds:
\begin{enumerate}[\upshape (1)]

\item $\jac_{x_1,x_2} \tilde{H}$ is a triangular matrix;

\item there are independent linear forms $a, b \in K[x]$,
such that
$$
\jac_{x_1,x_2} \tilde{H} =
\left( \begin{array}{cc} ab & -b^2 \\ a^2 & -ab \end{array} \right)
~~ \mbox{ and } ~~
\jac_{x_1,x_2} \left( \begin{array}{c} a \\ b \end{array} \right)
=  \left( \begin{array}{cc} 0 & 0 \\ 0 & 0 \end{array} \right)\mbox{;}
$$

\item $\textup{char} K=3$ and there are independent linear forms $a, b \in K[x]$,
such that
$$
\jac_{x_1,x_2} \tilde{H} =
\left( \begin{array}{cc} ab & -b^2 \\ a^2 & -ab \end{array} \right)
~~\mbox{and} ~~
\jac_{x_1,x_2} \left( \begin{array}{c} a \\ b \end{array} \right)
=  \left( \begin{array}{cc} 0 & 1 \\ 1 & 0 \end{array} \right)\mbox{.}
$$
\end{enumerate}
\end{lemma}

\begin{proof}
Suppose that (1) does not hold.
By Lemma \ref{2x2}, there are $a,b,c \in K[x]$, such that
$$
\jac_{x_1,x_2} H = c \left( \begin{array}{cc}
ab & - b^2 \\ a^2 & -ab \end{array} \right)
$$
where $a$ and $b$ are linearly independent over $K$. As $H$ is cubic
homogeneous, the entries of $\jac_{x_1,x_2} H$ are quadratic homogeneous,
so $c \in K$ and $a$ and $b$ are independent linear forms.

If we take
$$
T = \left( \begin{array}{cc} c & 0 \\ 0 & 1 \end{array} \right)
\mbox{,} \quad \mbox{then} \quad
\jac_{x_1,x_2} \tilde{H} = \left( \begin{array}{cc}
\tilde{a}\tilde{b} & -\tilde{b}^2 \\ \tilde{a}^2 & -\tilde{a}\tilde{b} \end{array} \right) \mbox{,}
$$
where $\tilde{a} = c \cdot a|_{x_1=cx_1}$ and $\tilde{b} =c^{-1}\cdot b|_{x_1=cx_1}$.

We claim that the coefficient $k_2$ of $x_2$ in $\tilde{b}$ is $0$. Suppose conversely that $k_2\neq 0$.  Then
the coefficient of $x_2^3$ in
$$
3 \tilde{H}_1 = \jac_{x_1,x_2} \tilde{H}_1 \cdot \binom{x_1}{x_2}=\tilde{b}(x_1\tilde{a}-x_2\tilde{b})
$$
is nonzero. In particular, $\chr K\neq 3$. One may verify that
$$
\jac_{x_1,x_2} (\tilde{H}_1 + \tfrac13 k_2^{-1}\tilde{b}^3) = (\tilde{c}\tilde{b},~0)\mbox{,}
$$
where $\tilde{c} := \tilde{a} + k_2^{-1}\tilde{b} (\parder{}{x_1} \tilde{b})$.
As a consequence, $\parder{}{x_2} (\tilde{c}\tilde{b}) = \parder{}{x_1} 0 = 0$. Furthermore,
$\tilde{c}$ and $\tilde{b}$ are independent, just like $\tilde{a}$ and
$\tilde{b}$.
By $\parder{}{x_2} (\tilde{c}\tilde{b}) = 0$, we have $\tilde{c}\tilde{b} \in
K[x_1,x_3,x_4,\ldots,x_n]$ if $\textup{char} K\neq 2$. Since $\tilde{c}$
and $\tilde{b}$ are independent, we deduce that if $\textup{char} K= 2$ then $\tilde{c}\tilde{b} \in
K[x_1,x_3,x_4,\ldots,x_n]$  as well.
Since the coefficient $\lambda$ of $x_2$ in $\tilde{b}$ is nonzero, we have
$\tilde{c} = 0$,  a contradiction.

So the coefficient of $x_2$ in $\tilde{b}$ is $0$. Similarly, the
coefficient of $x_1$ in $\tilde{a}$ is $0$. Consequently,
$$
\jac_{x_1,x_2} \left( \begin{array}{c} \tilde{a} \\ \tilde{b} \end{array} \right)
=  \left( \begin{array}{cc} 0 & \lambda \\ \mu & 0 \end{array} \right)
\mbox{,}$$
where $\lambda, \mu \in K$.
Therefore $$
\jac_{x_1,x_2} \tilde{H} = \left(\begin{array}{cc}
(\lambda x_2+\cdots)(\mu x_1+\cdots) & -(\mu x_1+\cdots)^2 \\
(\lambda x_2+\cdots)^2 & - (\lambda x_2+\cdots)(\mu x_1+\cdots)
\end{array} \right)\mbox{.}
$$
So the coefficient of $x_1^2 x_2$ in $2\tilde{H}_1$ is equal to both
$\lambda\mu$ and $-2\mu^2$. Similarly, the coefficient of $x_1 x_2^2$ in
$2\tilde{H}_2$ is equal to both $\lambda\mu$ and $-2\lambda^2$.
It follows that either
$\lambda = \mu = 0$ or $0\neq \lambda = -2\mu = 4\lambda$. In the former case, $\widetilde{H}$ satisfies (2). In the latter
case, $\textup{char}K=3$ and $\lambda = \mu$. Replacing $\tilde{H}$ by
$\lambda \tilde{H}\big(\lambda^{-1}(x_1,x_2),x_3,x_4,\ldots,x_n\big)$, we have that $\widetilde{H}$  satisfies (3).
\end{proof}

\begin{theorem} \label{uporkle2}
Suppose that $\chr K\neq 2, 3$.
Let $H \in K[x]^n$ be  cubic homogeneous such that $x+H$ is a Keller map, i.e., $\jac H$ is nilpotent.
\begin{enumerate}[\upshape (i)]

\item
If $\rk \jac H = 1$, then there exists $T \in \GL_n(K)$ such that for
$\tilde{H} := T^{-1} H(Tx)$,
\begin{align*}
\tilde{H}_1 &\in K[x_2,x_3,x_4,\ldots,x_n]\mbox{,}  \\
\tilde{H}_2 &= \tilde{H}_3 = \tilde{H}_4 = \cdots = \tilde{H}_n = 0\mbox{.}
\end{align*}

\item
If $\rk \jac H = 2$, then either $H$ is linearly triangularizable or there exists  $T \in \GL_n(K)$ such that for
$\tilde{H} := T^{-1} H(Tx)$,
\begin{align*}
\tilde{H}_1 &- (x_1x_3x_4-x_2x_4^2) \in K[x_3,x_4,\ldots,x_n]\mbox{,} \\
\tilde{H}_2 &- (x_1x_3^2-x_2x_3x_4) \in K[x_3,x_4,\ldots,x_n]\mbox{,} \\
\tilde{H}_3 &= \tilde{H}_4 = \cdots = \tilde{H}_n = 0\mbox{.}
\end{align*}
\end{enumerate}
Furthermore, $x+tH$ is invertible over $K[t]$ if $\rk \jac H\leq 2$, where $t$ is a new variable.
Moreover, $x+tH$ is even tame over $K[t]$ if either $\rk \jac H=1$ or $\rk \jac H=2$ and $n\neq 4$.
In particular, $x + \lambda H$ is invertible and tame under the above condition respectively
for every $\lambda \in K$.
\end{theorem}

\begin{proof}
We may take $\tilde{H}$ as in (1), (2) or (3) of Theorem \ref{rkle2}.
If $\rk \jac H = 1$, then $\tilde{H}$ is as in (1) of Theorem \ref{rkle2}, i.e., $\tilde{H}_{i}= 0, 2\leq i\leq n$,
whence (i) holds because $\tr \jac \tilde{H} = 0$. So assume that $\rk \jac H = 2$. Notice that $\jac H$ is nilpotent.

If $\tilde{H}$ is as in (1) or (2) of Theorem \ref{rkle2}, i.e., $\tilde{H}_{i}= 0, 3\leq i\leq n$ or $\tilde{H} \in K[x_1,x_2]^n$, then $\jac_{x_1,x_2}(\tilde{H}_1,\tilde{H}_2)$ is nilpotent.

If $\tilde{H}$ is as in (3) of Theorem \ref{rkle2}, i.e., $K \tilde{H}_1 + K \tilde{H}_2 + \cdots + K \tilde{H}_n
= K x_3 x_1^2 \oplus K x_3 x_1 x_2 \oplus K x_3 x_2^2$, then $\tilde{H}_3 = 0$,
because $x_3^{-1} \tilde{H}_3$ is the constant part with respect to
$x_3$ of $\tr \jac \tilde{H} = 0$. So $\jac_{x_1,x_2}(\tilde{H}_1,\tilde{H}_2)$ is nilpotent in any case.

One may observe that, in all the cases (1), (2) and (3) of Theorem \ref{rkle2}, if $\jac_{x_1,x_2}(\tilde{H}_1,\tilde{H}_2)$ is similar over
$K$ to a triangular matrix, then $\jac \tilde{H}$ is
similar over $K$ to a triangular matrix, and so is
$\jac H$, and thus $H$ is linearly triangularizable.

Now suppose $\jac_{x_1,x_2}(\tilde{H}_1,\tilde{H}_2)$
is not similar over $K$ to a triangular matrix.  Noticing that $\textup{char}K\neq 2,3$, $\jac_{x_1,x_2}(\tilde{H}_1,\tilde{H}_2)$ must be as in (2) of Lemma \ref{2x2nilp}, i.e., $$
\jac_{x_1,x_2} \tilde{H} =
\left( \begin{array}{cc} ab & -b^2 \\ a^2 & -ab \end{array} \right)
~~ \mbox{ and } ~~
\jac_{x_1,x_2} \left( \begin{array}{c} a \\ b \end{array} \right)
=  \left( \begin{array}{cc} 0 & 0 \\ 0 & 0 \end{array} \right)\mbox{,}
$$
where $a,b$ are linearly independent linear forms.

If $\tilde{H}_1 \in
K[x_1,x_2,x_3]$, then $a,b\in k[x_3]$, a contradiction. So $\tilde{H}$ is not as in (2) or (3) of Theorem
\ref{rkle2}, and thus is as in (1) of Theorem
\ref{rkle2}, i.e.,
$\tilde{H}_3 = \tilde{H}_4 = \cdots = \tilde{H}_n = 0$. Consequently, by linear coordinate transformation, we may take $\tilde{H}$ such that $a = x_3$
and $b = x_4$. So (ii) holds.

For the last claim, when $\rk \jac H=1$, $\widetilde{H}$ is of the form in (i), whence $x+t\widetilde{H}$ is elementary and thus tame. When $\rk \jac H=2$, $\widetilde{H}$ is of the form in (ii),
and it suffices to show the following automorphism
$$
F=\big(x_1 + tx_4(x_3x_1-x_4x_2),x_2 + tx_3(x_3x_1-x_4x_2),x_3,x_4,x_5\big)
$$
is tame over $K[t]$.

For that purpose, let $w=t(x_3x_1-x_4x_2)$ and let $D:=x_4\partial_{x_1}+x_3\partial_{x_2}$ be a derivation of $K[t][x_1,x_2,x_3,x_4]$.
Observe that $D$ is triangular and $w\in \ker D$, and that $F=(\textup{exp} (wD), x_5).$ Therefore $F$ is tame over $K[t]$ due to the following Lemma \ref{smithlemma}.
\end{proof}

Recall that a derivation $D$ of $K[x]$ is called locally nilpotent if for every $f\in K[x]$ there exists an $m$ such that $D^{m}(f)=0$.
For such a derivation, $\textup{exp} D:=\sum_{i=0}^\infty \frac{1}{i!}D^i$ is a polynomial automorphism of $K[x]$.
A derivation $D$ of $K[x]$ is called triangular if $D(x_i)\in K[x_{i+1},\ldots,x_n]$ for $i=1,2,\ldots,n-1$ and $D(x_n)\in K.$ A triangular derivation is locally nilpotent.

\begin{lemma}\label{smithlemma}  Let $D$ be a triangular derivation of $K[t][x]$ and $w\in \ker D$ i.e. $D(w)=0$. Then
$(\textup{exp}(wD),x_{n+1})$ is tame over $K[t]$.
\end{lemma}

\begin{proof} From \cite[Corollary]{smith}, it follows that 
there exists a $k$ such that $(\exp(wD),\allowbreak x_{n+1},x_{n+2},
\ldots,x_{n+k})$ is tame over $K(t)$. Inspecting the 
proof of \cite[Corollary]{smith} yields that $(\exp(wD),x_{n+1})$ is tame over $K[t]$. 
\end{proof}

\paragraph{Acknowledgments} 
The first author has been supported by the Netherlands Organisation of Scientific research (NWO).
The second author has been partially supported by the NSF of China (grant no. 11771176 and 11601146).


\end{document}